\documentclass{article}

\usepackage{amsmath}
\usepackage{amssymb}
\usepackage{amsthm}
\usepackage{color}
\usepackage{url}
\usepackage{fullpage}

\newtheorem{proposition}{Proposition}
\newtheorem{lemma}{Lemma}
\newtheorem{theorem}{Theorem}
\newtheorem{corollary}{Corollary}
\theoremstyle{definition}
\newtheorem{definition}{Definition}
\newtheorem{example}{Example}

\newcommand{\ZZ}{\mathbb{Z}}
\newcommand{\ZZp}{\mathbb{Z}_{>0}}
\newcommand{\ZZnn}{\mathbb{Z}_{\geq 0}}
\newcommand{\ord}{\mathsf{ord}}
\newcommand{\FF}{\mathcal{F}}
\newcommand{\MM}{\mathcal{M}}
\newcommand{\CC}{\mathcal{C}}
\newcommand{\Gen}{\mathsf{Gen}}
\newcommand{\Enc}{\mathsf{Enc}}
\newcommand{\Dec}{\mathsf{Dec}}
\newcommand{\Eval}{\mathsf{Eval}}
\newcommand{\pk}{\mathsf{pk}}
\newcommand{\sk}{\mathsf{sk}}
\newcommand{\ek}{\mathsf{ek}}
\newcommand{\PT}{\mathsf{PT}}

\begin{document}

\title{On Compression Functions over Groups with Applications to Homomorphic Encryption}
\author{Koji Nuida${}^{12}$\medskip\\
${}^1$ Institute of Mathematics for Industry (IMI), Kyushu University\\
\texttt{nuida@imi.kyushu-u.ac.jp}\\
${}^2$ National Institute of Advanced Industrial Science and Technology (AIST)}
\date{\today}

\maketitle

\begin{abstract}
Fully homomorphic encryption (FHE) enables an entity to perform arbitrary computation on encrypted data without decrypting the ciphertexts.
An ongoing group-theoretical approach to construct an FHE scheme uses a certain \lq\lq compression\rq\rq{} function $F(x)$ implemented by group operations on a given finite group $G$, which satisfies that $F(1) = 1$ and $F(\sigma) = F(\sigma^2) = \sigma$ where $\sigma \in G$ is some element of order $3$.
The previous work gave an example of such a function over the symmetric group $G = S_5$ by just a heuristic approach.
In this paper, we systematically study the possibilities of such a function over various groups.
We show that such a function does not exist over any solvable group $G$ (such as an Abelian group and a smaller symmetric group $S_n$ with $n \leq 4$).
We also construct such a function over the alternating group $G = A_5$ that has a shortest possible expression.
Moreover, by using this new function, we give a reduction of a construction of an FHE scheme to a construction of a homomorphic encryption scheme over the group $A_5$, which is more efficient than the previously known reductions.
\end{abstract}

\section{Introduction}
\label{sec:introduction}

\subsection{Background}
\label{sec:introduction__background}

\emph{Homomorphic encryption} is a special kind of public key encryption that enables an entity to perform \lq\lq computation on encrypted data\rq\rq{}; that is, given a certain $n$-ary operation $\varphi$ on plaintexts and ciphertexts $c_1,\dots,c_n$ with corresponding (unknown) plaintexts $m_1,\dots,m_n$, an entity can generate a ciphertext for plaintext $\varphi(m_1,\dots,m_n)$ without decrypting the input ciphertexts.
In particular, \emph{fully homomorphic encryption} (FHE) \cite{Gen09} can perform such computation on ciphertexts for an arbitrary operation $\varphi$ on plaintexts.
After the first construction of an FHE scheme by Gentry in 2009 \cite{Gen09}, many FHE schemes have been proposed in the literature, e.g., \cite{Bonte+22,BraVai11,Cheon+18,Chillotti+20,Dijk+10,DucMic15,GSW13,NuiKur15}.
Currently, all the known successful constructions of FHE schemes in the literature (with plausible security) follow basically the same framework using ciphertexts with \lq\lq noise\rq\rq{}.
That is, each ciphertext involves a noise term which grows when a homomorphic operation is applied, and a \lq\lq bootstrapping\rq\rq{} procedure \cite{Gen09} is executed to cancel out the noise term before it becomes too large to ensure correct decryption.
Usually, such a bootstrapping procedure is theoretically complicated, and is computationally much more expensive than the homomorphic evaluation of operations $\varphi$ themselves.

Besides the only successful approach so far to construction of FHE schemes described above, there is another ongoing approach from group theory.
An outline of the approach is as follows: (1) encode each plaintext bit $b \in \{0,1\}$ into an element $\sigma_b$ of a suitable (finite and) \emph{non-Abelian group} $G$; (2) implement an operation $\widetilde{\varphi}$ on the group $G$ using group operations of $G$, corresponding to a plaintext operation $\varphi$ via the encoding $b \mapsto \sigma_b$ (see below); (3) construct a homomorphic encryption scheme $\Pi_0$ with plaintext space $G$; (4) by using the functionality of $\Pi_0$, homomorphically evaluate the operation $\widetilde{\varphi}$.
For example, for the step (2) with $\varphi = \mathsf{AND}$, we require a function $\widetilde{\varphi}(x,y)$ that can be computed by using group operations of $G$ and satisfies that $\widetilde{\varphi}(\sigma_0,\sigma_0) = \widetilde{\varphi}(\sigma_0,\sigma_1) = \widetilde{\varphi}(\sigma_1,\sigma_0) = \sigma_0$ and $\widetilde{\varphi}(\sigma_1,\sigma_1) = \sigma_1$ (corresponding to $\mathsf{AND}(0,0) = \mathsf{AND}(0,1) = \mathsf{AND}(1,0) = 0$ and $\mathsf{AND}(1,1) = 1$).
To the author's best knowledge, such an approach was mentioned (informally) for the first time in \cite{GriPon04}, where the step (2) is supposed to be performed based on the results of \cite{BST90}.
The same approach was mentioned again in \cite{GriPon05}.
Later, such an approach was re-discovered (and mentioned informally) in \cite{OstSke08}, where the step (2) for the $\mathsf{NAND}$ operation is based on the properties of simple groups and commutators.
Moreover, a similar approach was also described in \cite{Nui21}, where the step (2) is performed in a more concrete manner.
The aim of the present paper is to analyze the approach of \cite{Nui21} in detail.
Here we emphasize that among the four steps above, the step (3) is obviously most difficult, and no successful solution for the step (3) has been given in the literature.
We note that a candidate homomorphic encryption scheme with plaintext space being the symmetric group $S_4$ was proposed in \cite{GriPon05}.
However, besides an issue that the ciphertexts in their scheme are not compact (i.e., the ciphertext size grows unboundedly by iterative homomorphic operations), there is an essential issue that the known constructions in the step (2) require the group $G$ to be non-solvable, while $S_4$ is a solvable group.

We explain an approach to the step (2) above in \cite{Nui21} called an \lq\lq approximate-then-adjust\rq\rq{} method.
In this approach, a target function $\widetilde{\varphi}$ is constructed by composition of a multivariate \lq\lq inner function\rq\rq{} $F^{\mathrm{in}}$ followed by a univariate \lq\lq outer function\rq\rq{} $F^{\mathrm{out}}$.
For example, for the case $\varphi = \mathsf{OR}$, we take an element $\sigma \in G$ of order $3$, set $\sigma_0 := 1$ and $\sigma_1 := \sigma$, and simply set $F^{\mathrm{in}}_{\mathsf{OR}}(x_1,x_2) := x_1 x_2$.
Now the three values $F^{\mathrm{in}}_{\mathsf{OR}}(\sigma_0,\sigma_0) = \sigma_0$, $F^{\mathrm{in}}_{\mathsf{OR}}(\sigma_0,\sigma_1) = \sigma_1$, and $F^{\mathrm{in}}_{\mathsf{OR}}(\sigma_1,\sigma_0) = \sigma_1$ correctly correspond to the values of $\mathsf{OR}$, while the remaining value $F^{\mathrm{in}}_{\mathsf{OR}}(\sigma_1,\sigma_1) = \sigma^2$ is not correct.
Then an outer function $F^{\mathrm{out}}$ satisfying that $F^{\mathrm{out}}(1) = 1$ and $F^{\mathrm{out}}(\sigma) = F^{\mathrm{out}}(\sigma^2) = \sigma$ can adjust the incorrect value (i.e., $F^{\mathrm{out}}( F^{\mathrm{in}}_{\mathsf{OR}}(\sigma_1,\sigma_1) ) = F^{\mathrm{out}}(\sigma^2) = \sigma_1$) while keeping the other correct values.
The same outer function $F^{\mathrm{out}}$ can be also used to realize some other operations; e.g., $\mathsf{NAND}$ with $F^{\mathrm{in}}_{\mathsf{NAND}}(x_1,x_2) = x_1{}^{-1} x_2{}^{-1} \sigma^2$ and $\mathsf{XOR}$ with $F^{\mathrm{in}}_{\mathsf{XOR}}(x_1,x_2) = x_1{}^{-1} x_2$.
Hence the problem is reduced to construct such a function $F^{\mathrm{out}}$ satisfying that $F^{\mathrm{out}}(1) = 1$ and $F^{\mathrm{out}}(\sigma) = F^{\mathrm{out}}(\sigma^2) = \sigma$.
In \cite{Nui21}, by setting $G$ to be the symmetric group $S_5$ and $\sigma = (1\ 2\ 3) \in S_5$, the following example of the function $F^{\mathrm{out}}$ was given:
\begin{equation}
\label{eq:known_construction}
F^{\mathrm{out}}(y) = (1\ 5)(2\ 3\ 4) \cdot y \cdot (2\ 3\ 4) \cdot y \cdot (3\ 4) \cdot y^2 \cdot (2\ 3)(4\ 5) \cdot y \cdot (2\ 3\ 4) \cdot y \cdot (3\ 4) \cdot y^2 \cdot (1\ 4\ 2\ 5) \enspace.
\end{equation}
However, this function was found by a heuristic argument, and no systematic approach to find such a function was given in \cite{Nui21}.
The aim of the present paper is to execute a systematic study for possibilities of such functions, possibly over smaller groups and having shorter expressions than \eqref{eq:known_construction}.

\subsection{Our Contributions}
\label{sec:introduction__contribution}

In this paper, we systematically study the possibilities of functions $F$ in some classes implemented on various groups $G$, including those satisfying the \lq\lq main condition\rq\rq{} that $F(1) = 1$ and $F(\sigma) = F(\sigma^2) = \sigma$ where $\sigma$ is some element of $G$ of order $3$.
One of our main results is that such a function with a certain property, including the case of the main condition, does not exist when $G$ is a solvable group (Theorem \ref{thm:solvable_group}).
As a consequence, the approach of \cite{Nui21} described in Section \ref{sec:introduction__background} cannot work if $G$ is a solvable group (such as an Abelian group and $S_4$).
Secondly, we show that if a function $F$ satisfying the main condition exists over $G = S_5$, then such a function also exists over the alternating group $G = A_5$ (Proposition \ref{prop:construction_in_S_5_is_reduced_to_A_5}).
Then, based on some other results, we performed a computer search for a function $F$ satisfying the main condition over $A_5$ and found such a function (Example \ref{exmp:over_A_5}).
This function is significantly simpler than \eqref{eq:known_construction}, and it is also shown (by Theorem \ref{thm:short_case_over_S_5}) that this function is with a shortest possible expression.
Moreover, by using this new function, we give a reduction of a construction of an FHE scheme to a construction of a homomorphic encryption scheme over the group $A_5$, which is more efficient than the previously known reductions in \cite{GriPon04,GriPon05,Nui21,OstSke08} (Theorem \ref{thm:FHE_from_A_5-HE}).

\section{Definitions and Basic Observations}
\label{sec:basic}

In this paper, we let $G$ be a finite group with unit element denoted by $1$.
Let $\ZZp$ and $\ZZnn$ denote the sets of positive integers and non-negative integers, respectively.
For any $n \in \ZZ$ and $k \in \ZZp$, let $n \bmod k$ denote the remainder of $n$ modulo $k$, taken from the interval $[0,k-1]$. 
For an element $g$ of a group, let $\ord(g)$ denote the order of $g$.
Let $x$ and $y_1,y_2,\dots$ denote variables not belonging to the group under consideration.
We also use the following terminology.

\begin{definition}
\label{defn:group_function}
We define a \emph{group function} over $G$ to be a sequence over $G \sqcup \{x\}$ of the form
\[
F(x) = g_0 x^{e_1} g_1 x^{e_2} \cdots g_{\ell-1} x^{e_{\ell}} g_{\ell}
\]
where $\ell \in \ZZp$, $g_i \in G$, $e_i \in \ZZp$, and $x^{e_i}$ is an abbreviation of $xx \cdots x$ ($e_i$ letters).
We call $\ell$ the \emph{size} of $F$ and call $(e_1,\dots,e_{\ell})$ the \emph{exponent} of $F$.
Moreover, we define the substitution of $h \in G$ into $F$ to be
\[
F(h) := g_0 h^{e_1} g_1 h^{e_2} \cdots g_{\ell-1} h^{e_{\ell}} g_{\ell} \in G \enspace.
\]
\end{definition}

The following is a main object of this paper.

\begin{definition}
\label{defn:compression_function}
Let $\sigma \in G \setminus \{1\}$ and $L \in \ZZp$, and let $\mu_j \in \ZZnn$ and $\rho_j \in G$ for each $j = 1,\dots,L$.
We define a \emph{compression function} of \emph{type $(\sigma; (\mu_1,\rho_1),\dots, (\mu_L,\rho_L))$} to be a group function $F$ satisfying that
\[
F(\sigma^{\mu_j}) = \rho_j \mbox{ for every } j = 1,\dots,L \enspace.
\]
\end{definition}

\begin{example}
\label{exmp:over_S_5}
Let $G$ be the symmetric group $S_5$ on five letters, and let $\sigma \in G$ be the cyclic permutation $(1\ 2\ 3)$.
It was found in \cite{Nui21} that the following group function of size $6$ and exponent $(1,1,2,1,1,2)$ over $G$,
\[
F(x) = (1\ 5)(2\ 3\ 4) \cdot x \cdot (2\ 3\ 4) \cdot x \cdot (3\ 4) \cdot x^2 \cdot (2\ 3)(4\ 5) \cdot x \cdot (2\ 3\ 4) \cdot x \cdot (3\ 4) \cdot x^2 \cdot (1\ 4\ 2\ 5)
\]
satisfies that $F(1) = 1$ and $F(\sigma) = F(\sigma^2) = \sigma$, therefore $F$ is a compression function over $G$ of type $(\sigma; (0,1), (1,\sigma), (2,\sigma))$.
In this example, the term \lq\lq compression function\rq\rq{} is motivated by the situation that the two-element set $\{\sigma,\sigma^2\}$ is compressed by the function $F$ to a single element $\sigma$. 
\end{example}

\begin{example}
\label{exmp:over_A_5}
We give another example smaller than Example \ref{exmp:over_S_5} as follows.
Let $G$ be the alternating group $A_5$ on five letters, and let $\sigma = (1\ 2\ 3) \in G$.
Then the following group function of size $4$ and exponent $(1,1,1,1)$ over $G$,
\[
F(x) = (1\ 2\ 4\ 3\ 5) \cdot x \cdot (1\ 3\ 5) \cdot x \cdot (1\ 4\ 3) \cdot x \cdot (15)(23) \cdot x \cdot (1\ 4\ 3\ 5\ 2)
\]
satisfies that $F(1) = 1$ and $F(\sigma) = F(\sigma^2) = \sigma$, therefore $F$ is also a compression function over $G$ of the same type $(\sigma; (0,1), (1,\sigma), (2,\sigma))$.
\end{example}

\begin{example}
\label{exmp:over_S_4}
We give examples of different types.
Let $G = S_4$, $\sigma = (1\ 2\ 3\ 4) \in G$, $\rho_1 = (2\ 4)$, $\rho_2 = (1\ 3\ 2\ 4)$, $\rho_3 = (1\ 2\ 3\ 4)$, $\rho_4 = (1\ 2)$.
Then the following group function of size $2$ and exponent $(1,1)$,
\[
F(x) = x \cdot (3\ 4) \cdot x \cdot (2\ 3\ 4)
\]
satisfies that $F(\sigma^{j-1}) = \rho_j$ for each $j \in \{1,2,3,4\}$, therefore $F$ is a compression function over $G$ of type $(\sigma; (0,\rho_1),(1,\rho_2),(2,\rho_3),(3,\rho_4))$.
(In this example, we still call the $F$ a compression function though it does not \lq\lq compress\rq\rq{} anything.)
Moreover, this implies that the following group function of size $4$ and exponent $(1,1,1,1)$,
\[
\widetilde{F}(x) = F(x)^2
= x \cdot (3\ 4) \cdot x \cdot (2\ 3\ 4) \cdot x \cdot (3\ 4) \cdot x \cdot (2\ 3\ 4)
\]
is a compression function of type $(\sigma; (0,\widetilde{\rho}_1),(1,\widetilde{\rho}_2),(2,\widetilde{\rho}_3),(3,\widetilde{\rho}_4))$ where $\widetilde{\rho}_j := \rho_j{}^2$.
Here we have $\widetilde{\rho}_1 = \widetilde{\rho}_4 = 1$, but $\widetilde{\rho}_2 = (1\ 2)(3\ 4)$ is not equal to $\widetilde{\rho}_3 = (1\ 3)(2\ 4)$.
In fact, it holds (by Theorem \ref{thm:solvable_group} below) that such a compression function over $G = S_4$ cannot exist if we moreover set $\widetilde{\rho}_2 = \widetilde{\rho}_3 \neq 1$.
\end{example}

In order to investigate (in)existence of compression functions, the following lemma is fundamental.

\begin{lemma}
\label{lem:restating_with_conjugate_elements_general}
Let $\sigma \in G \setminus \{1\}$ and $L \in \ZZp$, and let $\mu_j \in \ZZnn$ and $\rho_j \in G$ for $j = 1,\dots,L$.
Let $\ell \in \ZZp$ and $e_i \in \ZZp$ for $i = 1,\dots,\ell$.
Then the following conditions are equivalent:
\begin{enumerate}
\item \label{item:lem:restating_with_conjugate_elements_general__1}
There exists a compression function $F$ of type $(\sigma;(\mu_1,\rho_1),\dots,(\mu_L,\rho_L))$, size $\ell$, and exponent $(e_1,\dots,e_{\ell})$ over $G$.
\item \label{item:lem:restating_with_conjugate_elements_general__2}
The following system of equations over $G$,
\begin{equation}
\label{eq:lem:restating_with_conjugate_elements_general:equation}
y_1{}^{\mu_j e_1} y_2{}^{\mu_j e_2} \cdots y_{\ell}{}^{\mu_j e_{\ell}} y_{\ell+1} = \rho_j \mbox{ for } j = 1,\dots,L
\end{equation}
has a solution $(\tau_1,\dots,\tau_{\ell},\tau_{\ell+1}) \in G^{\ell+1}$ satisfying the following \emph{conjugacy condition}:
\begin{equation}
\label{eq:lem:restating_with_conjugate_elements_general:conjugacy}
\mbox{For each $i \in \{1,\dots,\ell\}$, $\tau_i$ is conjugate to $\sigma$ in $G$} \enspace.
\end{equation}
\end{enumerate}
\end{lemma}
\begin{proof}
First, we assume Condition \ref{item:lem:restating_with_conjugate_elements_general__1} and show that Condition \ref{item:lem:restating_with_conjugate_elements_general__2} holds.
Write
\[
F(x) = g_0 x^{e_1} g_1 x^{e_2} \cdots g_{\ell-1} x^{e_{\ell}} g_{\ell}
\]
with $g_0,\dots,g_{\ell} \in G$.
By putting $h_i := g_0 g_1 \cdots g_i$ for each $i = 0,1,\dots,\ell$, we have $g_i = h_{i-1}{}^{-1} h_i$ for any $i = 1,\dots,\ell$, therefore
\[
F(\nu) = h_0 \nu^{e_1} h_0{}^{-1} h_1 \nu^{e_2} h_1{}^{-1} \cdots h_{\ell-1} \nu^{e_{\ell}} h_{\ell-1}{}^{-1} \cdot h_{\ell} \mbox{ for any } \nu \in G \enspace.
\]
Moreover, for each $i = 1,\dots,\ell$, let $\tau_i := h_{i-1} \sigma h_{i-1}{}^{-1}$, which is conjugate to $\sigma$ in $G$.
Then for each $j = 1,\dots,L$, the condition $F(\sigma^{\mu_j}) = \rho_j$ in Definition \ref{defn:compression_function} implies that $\tau_1{}^{\mu_j e_1} \tau_2{}^{\mu_j e_2} \cdots \tau_{\ell}{}^{\mu_j e_{\ell}} h_{\ell} = \rho_j$.
Hence $(\tau_1,\dots,\tau_{\ell},h_{\ell})$ is a solution of the system of equations \eqref{eq:lem:restating_with_conjugate_elements_general:equation} satisfying the conjugacy condition \eqref{eq:lem:restating_with_conjugate_elements_general:conjugacy}, therefore Condition \ref{item:lem:restating_with_conjugate_elements_general__2} holds.

Conversely, we assume Condition \ref{item:lem:restating_with_conjugate_elements_general__2} and show that Condition \ref{item:lem:restating_with_conjugate_elements_general__1} holds.
As each $\tau_i$ for $i = 1,\dots,\ell$ is conjugate to $\sigma$ in $G$, we can write $\tau_i = h_{i-1} \sigma h_{i-1}{}^{-1}$ for some $h_{i-1} \in G$.
Now for each $j = 1,\dots,L$, the condition $\tau_1{}^{\mu_j e_1} \cdots \tau_{\ell}{}^{\mu_j e_{\ell}} \tau_{\ell+1} = \rho_j$ for the solution $(\tau_1,\dots,\tau_{\ell},\tau_{\ell+1})$ implies that
\[
h_0 (\sigma^{\mu_j})^{e_1} h_0{}^{-1} h_1 (\sigma^{\mu_j})^{e_2} h_1{}^{-1} \cdots h_{\ell-1} (\sigma^{\mu_j})^{e_{\ell}} h_{\ell-1}{}^{-1} \tau_{\ell+1} = \rho_j \enspace.
\]
Therefore, the group function
\[
F(x) := h_0 x^{e_1} (h_0{}^{-1} h_1) x^{e_2} (h_1{}^{-1} h_2) \cdots (h_{\ell-2}{}^{-1} h_{\ell-1}) x^{e_{\ell}} (h_{\ell-1}{}^{-1} \tau_{\ell+1})
\]
satisfies that $F(\sigma^{\mu_j}) = \rho_j$ for every $j = 1,\dots,L$.
Hence Condition \ref{item:lem:restating_with_conjugate_elements_general__1} holds.
This completes the proof.
\end{proof}

For this lemma, the equivalent condition can be slightly simplified when the type of a compression function is \lq\lq normalized\rq\rq{} by the condition $F(1) = 1$.
To state the result, we prepare the following terminology.

\begin{definition}
\label{defn:normalized_type}
Let $(\sigma; (\mu_1,\rho_1),\dots,(\mu_L,\rho_L))$ be a type of a compression function over $G$; that is, $\sigma \in G \setminus \{1\}$, $L \in \ZZp$, and for each $j = 1,\dots,L$, we have $\mu_j \in \ZZnn$ and $\rho_j \in G$.
We say that the type is \emph{normalized} if $\mu_1 = 0$ and $\rho_1 = 1 \in G$.
\end{definition}

\begin{lemma}
\label{lem:restating_with_conjugate_elements}
Let $(\sigma; (\mu_1,\rho_1),\dots,(\mu_L,\rho_L))$ be a normalized type of a compression function over $G$.
Let $\ell \in \ZZp$ and $e_i \in \ZZp$ for $i = 1,\dots,\ell$.
Then the following conditions are equivalent:
\begin{enumerate}
\item \label{item:lem:restating_with_conjugate_elements__1}
There exists a compression function $F$ of type $(\sigma; (\mu_1,\rho_1),\dots, (\mu_L,\rho_L))$, size $\ell$, and exponent $(e_1,\dots,e_{\ell})$ over $G$.
\item \label{item:lem:restating_with_conjugate_elements__2}
The following system of equations over $G$,
\begin{equation}
\label{eq:lem:restating_with_conjugate_elements:equation}
y_1{}^{\mu_j e_1} y_2{}^{\mu_j e_2} \cdots y_{\ell}{}^{\mu_j e_{\ell}} = \rho_j \mbox{ for } j = 2,\dots,L
\end{equation}
has a solution $(\tau_1,\dots,\tau_{\ell}) \in G^{\ell}$ satisfying the following \emph{conjugacy condition}:
\begin{equation}
\label{eq:lem:restating_with_conjugate_elements:conjugacy}
\mbox{For each $i \in \{1,\dots,\ell\}$, $\tau_i$ is conjugate to $\sigma$ in $G$} \enspace.
\end{equation}
\end{enumerate}
\end{lemma}
\begin{proof}
In Condition \ref{item:lem:restating_with_conjugate_elements_general__2} of Lemma \ref{lem:restating_with_conjugate_elements_general}, the \lq\lq normalized\rq\rq{} property $(\mu_1,\rho_1) = (0,1)$ implies that the equation \eqref{eq:lem:restating_with_conjugate_elements_general:equation} for $j = 1$ is satisfied if and only if $y_{\ell+1} = 1$.
Therefore the claim follows immediately from Lemma \ref{lem:restating_with_conjugate_elements_general}.
\end{proof}

\begin{example}
\label{exmp:over_A_5_equivalent}
Let $G = A_5$ and $\sigma = (1\ 2\ 3)$.
Put
\[
\tau_1 := (2\ 4\ 5) \,,\,
\tau_2 := (1\ 5\ 4) \,,\,
\tau_3 := (3\ 4\ 5) \,,\,
\tau_4 := (2\ 5\ 4) \enspace.
\]
Then we have
\[
\tau_1 \tau_2 \tau_3 \tau_4 = (1\ 2\ 3) = \sigma \,,\,
\tau_1{}^2 \tau_2{}^2 \tau_3{}^2 \tau_4{}^2 = (1\ 2\ 3) = \sigma \enspace.
\]
Therefore $(\tau_1,\tau_2,\tau_3,\tau_4) \in G^4$ is a solution of the system of equations
\[
\begin{cases}
y_1 y_2 y_3 y_4 = \sigma \enspace,\\
y_1{}^2 y_2{}^2 y_3{}^2 y_4{}^2 = \sigma \enspace.
\end{cases}
\]
Moreover, we have
\[
\begin{split}
\tau_1 &= (1\ 2\ 4\ 3\ 5) \sigma (1\ 2\ 4\ 3\ 5)^{-1} \,,\,
\tau_2 = (1\ 5\ 2\ 4\ 3) \sigma (1\ 5\ 2\ 4\ 3)^{-1} \enspace,\\
\tau_3 &= (1\ 3\ 5\ 2\ 4) \sigma (1\ 3\ 5\ 2\ 4)^{-1} \,,\,
\tau_4 = (1\ 2\ 5\ 3\ 4) \sigma (1\ 2\ 5\ 3\ 4)^{-1} \enspace.
\end{split}
\]
Therefore the solution $(\tau_1,\tau_2,\tau_3,\tau_4)$ satisfies the conjugacy condition \eqref{eq:lem:restating_with_conjugate_elements:conjugacy}.
Hence by Lemma \ref{lem:restating_with_conjugate_elements}, there exists a compression function $F$ of type $(\sigma; (0,1),(1,\sigma),(2,\sigma))$, size $4$, and exponent $(1,1,1,1)$ over $G$.
Now by following the proof of Lemma \ref{lem:restating_with_conjugate_elements}, we put
\[
\begin{split}
g_0 &:= (1\ 2\ 4\ 3\ 5) \,,\,
g_1 := (1\ 2\ 4\ 3\ 5)^{-1} (1\ 5\ 2\ 4\ 3)
= (1\ 3\ 5) \,,\,
g_2 := (1\ 5\ 2\ 4\ 3)^{-1} (1\ 3\ 5\ 2\ 4)
= (1\ 4\ 3) \enspace,\\
g_3 &:= (1\ 3\ 5\ 2\ 4)^{-1} (1\ 2\ 5\ 3\ 4)
= (1\ 5)(2\ 3) \,,\,
g_4 := (1\ 2\ 5\ 3\ 4)^{-1}
= (1\ 4\ 3\ 5\ 2) \enspace.
\end{split}
\]
Then the function $F$ is given by
\[
F(x) = g_0 x g_1 x g_2 x g_3 x g_4 \enspace,
\]
which is exactly the function given in Example \ref{exmp:over_A_5}.
\end{example}

\section{On Solutions with Commutativity Properties}
\label{sec:commutativity}

In this section, we show that under a certain condition, a compression function does not exist in some \lq\lq commutative\rq\rq{} cases.
The following is a key lemma of the argument.

\begin{lemma}
\label{lem:commutative_solutions}
Let $(\sigma; (\mu_1,\rho_1),\dots,(\mu_L,\rho_L))$ be a normalized type of a compression function over $G$ satisfying the following condition:
\begin{equation}
\label{eq:lem:commutative_solutions:type_condition}
\mbox{For some indices $j_1 \neq j_2$ with $\rho_{j_1} = \rho_{j_2}$, the value $\mu_{j_1} - \mu_{j_2}$ is not a multiple of $\ord(\rho_{j_1})$} \enspace.
\end{equation}
Then the system of equations \eqref{eq:lem:restating_with_conjugate_elements:equation} has no solution $(\tau_1,\dots,\tau_{\ell}) \in G^{\ell}$ satisfying that all $\tau_1,\dots,\tau_{\ell}$ commute with each other.
\end{lemma}
\begin{proof}
Assume for the contrary that such a solution $(\tau_1,\dots,\tau_{\ell})$ exists.
Then the commutativity property of $\tau_i$'s implies that $\tau_1{}^{\mu_j e_1} \cdots \tau_{\ell}{}^{\mu_j e_{\ell}} = \widetilde{\tau}{}^{\mu_j}$ for each $j$ where $\widetilde{\tau} := \tau_1{}^{e_1} \cdots \tau_{\ell}{}^{e_{\ell}}$, therefore we have $\widetilde{\tau}{}^{\mu_{j_1}} = \rho_{j_1} = \widetilde{\tau}{}^{\mu_{j_2}}$.
This implies that $\rho_{j_1}{}^{\mu_{j_1}} = \widetilde{\tau}{}^{\mu_{j_1} \mu_{j_2}} = \rho_{j_1}{}^{\mu_{j_2}}$ and therefore $\rho_{j_1}{}^{\mu_{j_1} - \mu_{j_2}} = 1$.
This contradicts the condition that $\mu_{j_1} - \mu_{j_2}$ is not a multiple of $\ord(\rho_{j_1})$.
Hence the claim holds. 
\end{proof}

A typical situation for condition \eqref{eq:lem:commutative_solutions:type_condition} is that those indices $j_1 \neq j_2$ with $\rho_{j_1} = \rho_{j_2}$ satisfy that $\rho_{j_1} \neq 1$ and $\mu_{j_1} = \mu_{j_2} \pm 1$.
For example, the normalized type $(\sigma; (0,1), (1,\sigma), (2,\sigma))$ in Examples \ref{exmp:over_S_5} and \ref{exmp:over_A_5} satisfies condition \eqref{eq:lem:commutative_solutions:type_condition}.

We have the following consequences of Lemma \ref{lem:commutative_solutions}.

\begin{proposition}
\label{prop:type_1_size_one}
Let $(\sigma; (\mu_1,\rho_1),\dots,(\mu_L,\rho_L))$ be a normalized type of a compression function satisfying condition \eqref{eq:lem:commutative_solutions:type_condition}.
Then there does not exist a compression function $F$ of type $(\sigma; (\mu_1,\rho_1),\dots,(\mu_L,\rho_L))$ and size $1$ over $G$.
\end{proposition}
\begin{proof}
As the size of $F$ is $\ell = 1$, the commutativity condition for a solution in Lemma \ref{lem:commutative_solutions} is automatically satisfied.
Therefore by Lemma \ref{lem:commutative_solutions}, the system of equations \eqref{eq:lem:restating_with_conjugate_elements:equation} in Lemma \ref{lem:restating_with_conjugate_elements} has no solution.
Hence the claim follows from Lemma \ref{lem:restating_with_conjugate_elements}.
\end{proof}

\begin{proposition}
\label{prop:type_1_Abelian}
Let $(\sigma; (\mu_1,\rho_1),\dots,(\mu_L,\rho_L))$ be a normalized type of a compression function satisfying condition \eqref{eq:lem:commutative_solutions:type_condition}.
Suppose moreover that all elements of $G$ conjugate to $\sigma$ commute with each other.
Then there does not exist a compression function $F$ of type $(\sigma; (\mu_1,\rho_1),\dots,(\mu_L,\rho_L))$ over $G$.
In particular, if $G$ is Abelian, then there does not exist a compression function $F$ of type $(\sigma; (\mu_1,\rho_1),\dots,(\mu_L,\rho_L))$ over $G$.
\end{proposition}
\begin{proof}
By the assumption that all elements of $G$ conjugate to $\sigma$ commute with each other, the conjugacy condition \eqref{eq:lem:restating_with_conjugate_elements:conjugacy} in Lemma \ref{lem:restating_with_conjugate_elements} implies the commutativity condition in Lemma \ref{lem:commutative_solutions}.
Therefore by Lemma \ref{lem:commutative_solutions}, the system of equations \eqref{eq:lem:restating_with_conjugate_elements:equation} in Lemma \ref{lem:restating_with_conjugate_elements} has no solution satisfying the conjugacy condition \eqref{eq:lem:restating_with_conjugate_elements:conjugacy}.
Hence the claim follows from Lemma \ref{lem:restating_with_conjugate_elements}.
\end{proof}

We also have a result similar to Lemma \ref{lem:commutative_solutions} as follows.

\begin{lemma}
\label{lem:partly_commutative_solutions}
Let $(\sigma; (\mu_1,\rho_1),\dots,(\mu_L,\rho_L))$ be a normalized type of a compression function over $G$ satisfying the following condition:
\begin{equation}
\label{eq:lem:partly_commutative_solutions:type_condition}
\mbox{For some indices $j_1 \neq j_2$ with $\rho_{j_1} = \rho_{j_2} \neq 1$, the value $\mu_{j_1} - \mu_{j_2}$ divides $\mu_{j_1}$} \enspace.
\end{equation}
Then the system of equations \eqref{eq:lem:restating_with_conjugate_elements:equation} has no solution $(\tau_1,\dots,\tau_{\ell}) \in G^{\ell}$ satisfying either that $\tau_1,\dots,\tau_{\ell-1}$ commute with each other, or that $\tau_2,\dots,\tau_{\ell}$ commute with each other.
\end{lemma}
\begin{proof}
Assume for the contrary that such a solution $(\tau_1,\dots,\tau_{\ell})$ exists.
Put
\[
\begin{cases}
\widetilde{\tau} := \tau_1{}^{e_1} \cdots \tau_{\ell-1}{}^{e_{\ell-1}} \mbox{ and } \widehat{\tau} := \tau_{\ell}{}^{e_{\ell}} & \mbox{if $\tau_1,\dots,\tau_{\ell-1}$ commute} \enspace,\\
\widetilde{\tau} := \tau_1{}^{e_1} \mbox{ and } \widehat{\tau} := \tau_2{}^{e_2} \cdots \tau_{\ell}{}^{e_{\ell}} & \mbox{if $\tau_2,\dots,\tau_{\ell}$ commute} \enspace.
\end{cases}
\]
Then we have $\tau_1{}^{\mu_j e_1} \cdots \tau_{\ell}{}^{\mu_j e_{\ell}} = \widetilde{\tau}{}^{\mu_j} \widehat{\tau}{}^{\mu_j}$ for each $j$, therefore $\widetilde{\tau}{}^{\mu_{j_1}} \widehat{\tau}{}^{\mu_{j_1}} = \rho_{j_1} = \widetilde{\tau}{}^{\mu_{j_2}} \widehat{\tau}{}^{\mu_{j_2}}$.
This implies that $\widetilde{\tau}{}^{\mu_{j_1} - \mu_{j_2}} = \widehat{\tau}{}^{\mu_{j_2} - \mu_{j_1}}$.
By taking a $c \in \ZZ$ with $\mu_{j_1} = c (\mu_{j_1} - \mu_{j_2})$, it follows that $\widetilde{\tau}{}^{c(\mu_{j_1} - \mu_{j_2})} = \widehat{\tau}{}^{c(\mu_{j_2} - \mu_{j_1})}$, that is, $\widetilde{\tau}{}^{\mu_{j_1}} = \widehat{\tau}{}^{-\mu_{j_1}}$.
Therefore we have $\rho_{j_1} = \widetilde{\tau}{}^{\mu_{j_1}} \widehat{\tau}{}^{\mu_{j_1}} = 1$, contradicting the assumption that $\rho_{j_1} \neq 1$.
Hence the claim holds.
\end{proof}

We have the following consequences of Lemma \ref{lem:partly_commutative_solutions}.

\begin{proposition}
\label{prop:type_2_size_one}
Let $(\sigma; (\mu_1,\rho_1),\dots,(\mu_L,\rho_L))$ be a normalized type of a compression function satisfying condition \eqref{eq:lem:partly_commutative_solutions:type_condition}.
Then there does not exist a compression function $F$ of type $(\sigma; (\mu_1,\rho_1),\dots,(\mu_L,\rho_L))$ and size at most $2$ over $G$.
\end{proposition}
\begin{proof}
As the size of $F$ is $\ell \leq 2$, the commutativity condition for a solution in Lemma \ref{lem:partly_commutative_solutions} is automatically satisfied.
Therefore by Lemma \ref{lem:partly_commutative_solutions}, the system of equations \eqref{eq:lem:restating_with_conjugate_elements:equation} in Lemma \ref{lem:restating_with_conjugate_elements} has no solution.
Hence the claim follows from Lemma \ref{lem:restating_with_conjugate_elements}.
\end{proof}

\begin{proposition}
\label{prop:type_2_Abelian}
Let $(\sigma; (\mu_1,\rho_1),\dots,(\mu_L,\rho_L))$ be a normalized type of a compression function satisfying condition \eqref{eq:lem:partly_commutative_solutions:type_condition}.
Suppose moreover that all elements of $G$ conjugate to $\sigma$ commute with each other.
Then there does not exist a compression function $F$ of type $(\sigma; (\mu_1,\rho_1),\dots,(\mu_L,\rho_L))$ over $G$.
In particular, if $G$ is Abelian, then there does not exist a compression function $F$ of type $(\sigma; (\mu_1,\rho_1),\dots,(\mu_L,\rho_L))$ over $G$.
\end{proposition}
\begin{proof}
By the assumption that all elements of $G$ conjugate to $\sigma$ commute with each other, the conjugacy condition \eqref{eq:lem:restating_with_conjugate_elements:conjugacy} in Lemma \ref{lem:restating_with_conjugate_elements} implies the commutativity condition in Lemma \ref{lem:partly_commutative_solutions}.
Therefore by Lemma \ref{lem:partly_commutative_solutions}, the system of equations \eqref{eq:lem:restating_with_conjugate_elements:equation} in Lemma \ref{lem:restating_with_conjugate_elements} has no solution satisfying the conjugacy condition \eqref{eq:lem:restating_with_conjugate_elements:conjugacy}.
Hence the claim follows from Lemma \ref{lem:restating_with_conjugate_elements}.
\end{proof}

\section{On Normal Subgroups and Quotients}
\label{sec:normal_subgroups}

In this section, we investigate some reductions of the search for compression functions to smaller cases of normal subgroups and quotient groups.
We start with the following easy lemma.

\begin{lemma}
\label{lem:reduction_by_homomorphism}
Let $G'$ be a finite group and $\varphi \colon G \to G'$ a group homomorphism.
Let $F(x) = g_0 x^{e_1} \cdots g_{\ell-1} x^{e_{\ell}} g_{\ell}$ be a compression function of type $(\sigma; (\mu_1,\rho_1),\dots,(\mu_L,\rho_L))$ over $G$, and suppose that $\rho(\sigma) \neq 1$.
Then $\overline{F}(x) := \varphi(g_0) x^{e_1} \cdots \varphi(g_{\ell-1}) x^{e_{\ell}} \varphi(g_{\ell})$ is a compression function of type $(\varphi(\sigma); (\mu_1,\varphi(\rho_1)),\dots,(\mu_L,\varphi(\rho_L)))$ over $G'$.
\end{lemma}

Then we have the following consequences of Lemma \ref{lem:reduction_by_homomorphism} about reductions to the cases of quotient groups.

\begin{corollary}
\label{cor:over_direct_product}
Suppose that $G$ is a direct product of groups $G = H_1 \times H_2$.
For each $k = 1,2$, let $\pi_k \colon G \to H_k$ be the natural projection.
If there exists a compression function of type $(\sigma,(\mu_1,\rho_1),\dots,(\mu_L,\rho_L))$, size $\ell$, and exponent $(e_1,\dots,e_{\ell})$ over $G$, then for some $k \in \{1,2\}$, there exists a compression function of type $(\pi_k(\sigma),(\mu_1,\pi_k(\rho_1)),\dots,(\mu_L,\pi_k(\rho_L)))$, size $\ell$, and exponent $(e_1,\dots,e_{\ell})$ over $H_k$.
\end{corollary}
\begin{proof}
This follows from Lemma \ref{lem:reduction_by_homomorphism} and the fact that $\pi_k(\sigma) \neq 1$ for some $k \in \{1,2\}$ since $\sigma \neq 1$.
\end{proof}

\begin{corollary}
\label{cor:reduction_to_quotient}
Let $N$ be a normal subgroup of $G$.
Let $\pi \colon G \to G/N$ be the natural projection.
If there exists a compression function of type $(\sigma,(\mu_1,\rho_1),\dots,(\mu_L,\rho_L))$, size $\ell$, and exponent $(e_1,\dots,e_{\ell})$ over $G$, and $|N| \not\equiv 0 \pmod{\ord(\sigma)}$, then there exists a compression function of type $(\pi(\sigma),(\mu_1,\pi(\rho_1)),\dots,(\mu_L,\pi(\rho_L)))$, size $\ell$, and exponent $(e_1,\dots,e_{\ell})$ over $G/N$.
\end{corollary}
\begin{proof}
We have $\sigma \not\in N$ by the assumption that $|N| \not\equiv 0 \pmod{\ord(\sigma)}$.
Hence the claim follows from Lemma \ref{lem:reduction_by_homomorphism}.
\end{proof}

We also have the following lemma about reductions to the cases of quotient groups.

\begin{lemma}
\label{lem:with_normal_having_condition_C}
Let $(\sigma; (\mu_1,\rho_1),\dots,(\mu_L,\rho_L))$ be a normalized type of a compression function over $G$ satisfying condition \eqref{eq:lem:commutative_solutions:type_condition} or condition \eqref{eq:lem:partly_commutative_solutions:type_condition}.
Let $N$ be a normal subgroup of $G$ satisfying that all elements of $N$ having the same order as $\sigma$ commute with each other.
Let $\pi \colon G \to G/N$ be the natural projection.
If there exists a compression function $F$ of type $(\sigma; (\mu_1,\rho_1),\dots,(\mu_L,\rho_L))$, size $\ell$, and exponent $(e_1,\dots,e_{\ell})$ over $G$, then there exists a compression function of type $(\pi(\sigma); (\mu_1,\pi(\rho_1)),\dots,(\mu_L,\pi(\rho_L)))$, size $\ell$, and exponent $(e_1,\dots,e_{\ell})$ over $G/N$.
\end{lemma}
\begin{proof}
If $\sigma \not\in N$, then Lemma \ref{lem:reduction_by_homomorphism} applied to the natural projection $\pi$ implies the claim.
Therefore it suffices to deduce a contradiction by assuming that $\sigma \in N$.
Let $(\tau_1,\dots,\tau_{\ell})$ be a solution of the system of equations \eqref{eq:lem:restating_with_conjugate_elements:equation} as in Lemma \ref{lem:restating_with_conjugate_elements} implied by the existence of $F$.
As each $\tau_i$ is conjugate to $\sigma \in N$ and $N \trianglelefteq G$, we have $\tau_i \in N$ and $\ord(\tau_i) = \ord(\sigma)$ for every $i$.
Hence we have $\rho_j \in N$ for every $j$, therefore $(\sigma; (\mu_1,\rho_1),\dots,(\mu_L,\rho_L))$ is a normalized type of a compression function over $N$ satisfying condition \eqref{eq:lem:commutative_solutions:type_condition} or condition \eqref{eq:lem:partly_commutative_solutions:type_condition}.
Moreover, by the assumption on $N$, $(\tau_1,\dots,\tau_{\ell})$ is a solution of the system of equations \eqref{eq:lem:restating_with_conjugate_elements:equation} over $N$ satisfying that all $\tau_1,\dots,\tau_{\ell}$ commute with each other.
This contradicts Lemma \ref{lem:commutative_solutions} and Lemma \ref{lem:partly_commutative_solutions}.
Hence the claim holds.
\end{proof}

Then we have the following consequence of Lemma \ref{lem:with_normal_having_condition_C} about reductions to the cases of quotient groups.

\begin{corollary}
\label{cor:center_in_the_normal_subgroup}
Let $(\sigma; (\mu_1,\rho_1),\dots,(\mu_L,\rho_L))$ be a normalized type of a compression function over $G$ satisfying condition \eqref{eq:lem:commutative_solutions:type_condition} or condition \eqref{eq:lem:partly_commutative_solutions:type_condition}.
Let $N$ be a normal subgroup of $G$ satisfying that the center $Z(N)$ of $N$ is non-trivial.
Then we have $Z(N) \trianglelefteq G$, and if there exists a compression function $F$ of type $(\sigma; (\mu_1,\rho_1),\dots,(\mu_L,\rho_L))$, size $\ell$, and exponent $(e_1,\dots,e_{\ell})$ over $G$, then there exists a compression function of type $(\pi(\sigma); (\mu_1,\pi(\rho_1)),\dots,(\mu_L,\pi(\rho_L)))$, size $\ell$, and exponent $(e_1,\dots,e_{\ell})$ over $G/Z(N)$ where $\pi \colon G \to G/Z(N)$ is the natural projection.
\end{corollary}
\begin{proof}
First, as $Z(N)$ is preserved by any group automorphism on $N$, every inner automorphism of $G$ (which induces an automorphism on the normal subgroup $N$) also preserves $Z(N)$.
Therefore we have $Z(N) \trianglelefteq G$.
Now the claim follows from Lemma \ref{lem:with_normal_having_condition_C}, as the Abelian group $Z(N)$ satisfies the assumption in Lemma \ref{lem:with_normal_having_condition_C} that all elements of $Z(N)$ having the same order as $\sigma$ commute with each other.
\end{proof}

On the other hand, we have the following result about reductions to the cases of normal subgroups.

\begin{proposition}
\label{prop:reduction_to_normal_subgroup}
Let $(\sigma; (\mu_1,\rho_1),\dots,(\mu_L,\rho_L))$ be a normalized type of a compression function over $G$.
Let $N$ be a normal subgroup of $G$ satisfying that $\sigma \in N$ and $G = N Z_G(\sigma)$ where $Z_G(\sigma)$ denotes the centralizer of $\sigma$ in $G$.
If there exists a compression function $F$ of type $(\sigma; (\mu_1,\rho_1),\dots,(\mu_L,\rho_L))$, size $\ell$, and exponent $(e_1,\dots,e_{\ell})$ over $G$, then we have $\rho_j \in N$ for every $j$, and there exists a compression function of type $(\sigma; (\mu_1,\rho_1),\dots,(\mu_L,\rho_L))$, size $\ell$, and exponent $(e_1,\dots,e_{\ell})$ over $N$.
\end{proposition}
\begin{proof}
Let $(\tau_1,\dots,\tau_{\ell})$ be a solution of the system of equations \eqref{eq:lem:restating_with_conjugate_elements:equation} over $G$ as in Lemma \ref{lem:restating_with_conjugate_elements} implied by the existence of $F$.
For each $i$, write $\tau_i = u_i \sigma u_i{}^{-1}$ with $u_i \in G$.
Now by the assumption that $N \trianglelefteq G$ and $G = N Z_G(\sigma)$, we can write $u_i$ as $u_i = h_i z_i$ with $h_i \in N$ and $z_i \in Z_G(\sigma)$.
Then we have $\tau_i = h_i z_i \sigma z_i{}^{-1} h_i{}^{-1} = h_i \sigma h_i{}^{-1}$, that is, $\tau_i$ is an element of $N$ conjugate to $\sigma$ in $N$.
This implies that $\rho_j \in N$ for every $j$ and the system of equations \eqref{eq:lem:restating_with_conjugate_elements:equation} over $N$ has a solution $(\tau_1,\dots,\tau_{\ell})$ satisfying the conjugacy condition \eqref{eq:lem:restating_with_conjugate_elements:conjugacy}.
Therefore the claim follows from Lemma \ref{lem:restating_with_conjugate_elements}.
\end{proof}

\section{Inexistence over Some Classes of Groups}
\label{sec:inexistence}

In this section, we show inexistence of compression functions over some classes of groups.
The following is a key lemma of the argument.

\begin{lemma}
\label{lem:inexistence_by_normal}
Let $(\sigma; (\mu_1,\rho_1),\dots,(\mu_L,\rho_L))$ be a normalized type of a compression function over $G$ satisfying the following condition (*):
\begin{quote}
\textbf{Condition (*)}\quad
There exist indices $j_1 \neq j_2$ with $\rho_{j_1} = \rho_{j_2}$ satisfying that $\rho_{j_1} = \sigma^c$ for some integer $c$ coprime to $\ord(\sigma)$ and one of the following two conditions holds:
\begin{enumerate}
\item
$\mu_{j_1} - \mu_{j_2}$ is coprime to $\ord(\sigma)$;
\item
$\mu_{j_1} - \mu_{j_2}$ divides $\mu_{j_1}$.
\end{enumerate}
\end{quote}
Let $H$ be a subgroup of $G$ and $N$ a normal subgroup of $H$.
Suppose that $\sigma \in H \setminus N$.
Let $\pi \colon H \to H/N$ be the natural projection.
Suppose moreover that the following two conditions are satisfied:
\begin{quote}
\textbf{Condition (C1)}\quad
Any element of $G$ conjugate to $\sigma$ belongs to $H$.
\end{quote} 
\begin{quote}
\textbf{Condition (C2)}\quad
If $\nu_1,\nu_2 \in H$ and $\ord(\nu_1) = \ord(\nu_2) = \ord(\sigma)$, then $\pi(\nu_1) \pi(\nu_2) = \pi(\nu_2) \pi(\nu_1)$.
\end{quote}
Then there does not exist a compression function $F$ of type $(\sigma; (\mu_1,\rho_1),\dots,(\mu_L,\rho_L))$ over $G$.
\end{lemma}
\begin{proof}
By the assumption that $\sigma \in H \setminus N$, we have $\pi(\sigma) \neq 1$.
Moreover, as $c$ is coprime to $\ord(\sigma)$, we have $\ord(\rho_{j_1}) = \ord(\sigma)$ and $\sigma \in \langle \sigma^c \rangle = \langle \rho_{j_1} \rangle$, therefore $\rho_{j_1}$ is also an element of $H \setminus N$, hence $\pi(\rho_{j_1}) \neq 1$.

Now assume for the contrary that such a compression function $F$ exists.
Let $(\tau_1,\dots,\tau_{\ell})$ be a solution of the system of equations \eqref{eq:lem:restating_with_conjugate_elements:equation} as in Lemma \ref{lem:restating_with_conjugate_elements} implied by the existence of $F$.
Then by condition (C1), the conjugacy condition \eqref{eq:lem:restating_with_conjugate_elements:conjugacy} implies that each $\tau_i$ also belongs to $H$ and satisfies that $\ord(\tau_i) = \ord(\sigma)$.
Therefore each $\rho_j$ belongs to $H$ as well.
Now $(\pi(\tau_1),\dots,\pi(\tau_{\ell}))$ is a solution of the system of equations \eqref{eq:lem:restating_with_conjugate_elements:equation} over $H / N$ where $\pi(\rho_j)$ plays the role of $\rho_j$.
Moreover, as each $\tau_i$ satisfies that $\ord(\tau_i) = \ord(\sigma)$ as mentioned above, by condition (C2), it follows that all $\pi(\tau_1),\dots,\pi(\tau_{\ell})$ commute with each other.
Furthermore, if the type of $F$ satisfies the first condition in condition (*), then we have $\pi(\rho_{j_1}) = \pi(\rho_{j_2}) \neq 1$ and the value $\mu_{j_1} - \mu_{j_2}$ is not a multiple of $\ord(\pi(\rho_{j_1}))$ as $\ord(\pi(\rho_{j_1})) \neq 1$ is a divisor of $\ord(\rho_{j_1}) = \ord(\sigma)$.
Therefore $(\pi(\sigma); (\mu_1,\pi(\rho_1)),\dots,(\mu_L,\pi(\rho_L)))$ is a normalized type satisfying condition \eqref{eq:lem:commutative_solutions:type_condition}.
This contradicts Lemma \ref{lem:commutative_solutions}.
On the other hand, if the type of $F$ satisfies the second condition in condition (*), then we have $\pi(\rho_{j_1}) = \pi(\rho_{j_2}) \neq 1$ and the value $\mu_{j_1} - \mu_{j_2}$ divides $\mu_{j_1}$.
Therefore $(\pi(\sigma); (\mu_1,\pi(\rho_1)),\dots,(\mu_L,\pi(\rho_L)))$ is a normalized type satisfying condition \eqref{eq:lem:partly_commutative_solutions:type_condition}.
This contradicts Lemma \ref{lem:partly_commutative_solutions}.
Hence we have a contradiction in any case.
This completes the proof.
\end{proof}

By using this lemma, we show that under the condition (*), a compression function over a solvable group does not exist.

\begin{theorem}
\label{thm:solvable_group}
Suppose that $G$ is solvable.
Let $(\sigma; (\mu_1,\rho_1),\dots,(\mu_L,\rho_L))$ be a normalized type of a compression function over $G$ satisfying condition (*) in Lemma \ref{lem:inexistence_by_normal}.
Then there does not exist a compression function of type $(\sigma; (\mu_1,\rho_1),\dots,(\mu_L,\rho_L))$ over $G$.
\end{theorem}
\begin{proof}
Let $G = G^{(0)} > G^{(1)} > \cdots > G^{(n-1)} > G^{(n)} = \{1\}$ ($G^{(k)} = [G^{(k-1)},G^{(k-1)}]$) be the derived series of the solvable group $G$.
Note that it holds by induction on $k$ that each $G^{(k)}$ is normal in $G$.
As $\sigma \neq 1$, there exists an index $k < n$ satisfying that $\sigma \in G^{(k)} \setminus G^{(k+1)}$.
Put $H := G^{(k)}$ and $N := G^{(k+1)}$.
Then condition (C1) in Lemma \ref{lem:inexistence_by_normal} is satisfied as $G^{(k)} \trianglelefteq G$, and condition (C2) in Lemma \ref{lem:inexistence_by_normal} is satisfied as $G^{(k)} / G^{(k+1)}$ is Abelian.
Therefore all assumptions in Lemma \ref{lem:inexistence_by_normal} are satisfied.
Hence the claim follows from Lemma \ref{lem:inexistence_by_normal}.
\end{proof}

A typical situation for the condition (*) is that $\mu_2 = \mu_3 \pm 1$ and $\rho_2 = \rho_3 = \sigma$, including the type $(\sigma; (0,1), (1,\sigma), (2,\sigma))$ in Examples \ref{exmp:over_S_5} and \ref{exmp:over_A_5}.
Therefore by Theorem \ref{thm:solvable_group}, a compression function of such a type over a solvable group does not exist.

We also consider another situation as in Theorem \ref{thm:normal_series_starting_from_commutative_group} below.
Here we use the following lemma.

\begin{lemma}
\label{lem:elements_of_order_three_in_normal_series}
Let $p$ be a prime, and suppose that $G$ has a subnormal series $G = G_0 \rhd G_1 \rhd \cdots \rhd G_{n-1} \rhd G_n$ with $n \geq 1$ satisfying that $|G| / |G_n| \not\equiv 0 \pmod{p}$.
Then any element of order $p$ in $G$ is involved in $G_n$.
\end{lemma}
\begin{proof}
First note that $|G| / |G_n| = (|G| / |G_1|) \cdot (|G_1| / |G_n|) \not\equiv 0 \pmod{p}$ by the assumption, therefore we have $|G| / |G_1| \not\equiv 0 \pmod{p}$ and $|G_1| / |G_n| \not\equiv 0 \pmod{p}$.
Hence the claim for a general $n$ follows recursively from the claim for $n = 1$.
For the case $n = 1$, take a Sylow $p$-subgroup $P$ of $G_1$.
Then by the assumption that $|G| / |G_1| \not\equiv 0 \pmod{p}$, $P$ is also a Sylow $p$-subgroup of $G$.
Now for any $\nu \in G$ of order $p$, by Sylow's Theorem, $\nu$ is involved in some Sylow $p$-subgroup $P'$ of $G$, and $P'$ is conjugate to $P$.
As $P \leq G_1 \lhd G$, this implies that $P' \leq G_1$, therefore we have $\nu \in G_1$, as desired.
Hence the claim holds.
\end{proof}

\begin{theorem}
\label{thm:normal_series_starting_from_commutative_group}
Let $(\sigma; (\mu_1,\rho_1),\dots,(\mu_L,\rho_L))$ be a normalized type of a compression function over $G$ satisfying condition (*) in Lemma \ref{lem:inexistence_by_normal}.
Let $p$ be a prime and suppose that $\ord(\sigma) = p$.
Suppose moreover that $G$ has a subnormal series $G = G_0 \rhd G_1 \rhd \cdots \rhd G_{n-1} \rhd G_n$ with $n \geq 1$ satisfying that $|G| / |G_{n-1}| \not\equiv 0 \pmod{p}$, $|G_n| \not\equiv 0 \pmod{p}$, and all elements of $G_{n-1} / G_n$ of order $p$ commute with each other.
Then there does not exist a compression function of type $(\sigma; (\mu_1,\rho_1),\dots,(\mu_L,\rho_L))$ over $G$.
\end{theorem}
\begin{proof}
Put $H := G_{n-1}$ and $N := G_n$.
Let $\pi \colon G_{n-1} \to G_{n-1} / G_n$ be the natural projection.
By Lemma \ref{lem:elements_of_order_three_in_normal_series} applied to the subnormal series $G \rhd G_1 \rhd \cdots \rhd G_{n-1}$, it follows that any element of order $p$ in $G$ is involved in $G_{n-1}$.
In particular, we have $\sigma \in G_{n-1}$ and any element of $G$ conjugate to $\sigma$ is also an element of $G_{n-1}$.
Therefore condition (C1) in Lemma \ref{lem:inexistence_by_normal} is satisfied.
On the other hand, by the assumption that $|G_n| \not\equiv 0 \pmod{p}$, we have $\sigma \not\in G_n$.
Moreover, for any $\nu_1,\nu_2 \in G_{n-1}$ with $\ord(\nu_1) = \ord(\nu_2) = p$, as $p$ is prime, it holds that either at least one of $\pi(\nu_1)$ and $\pi(\nu_2)$ is a unit element, or $\ord(\pi(\nu_1)) = \ord(\pi(\nu_2)) = p$.
In any case, we have $\pi(\nu_1) \pi(\nu_2) = \pi(\nu_2) \pi(\nu_1)$ by the assumption that all elements of $G_{n-1} / G_n$ of order $p$ commute with each other.
Hence condition (C2) in Lemma \ref{lem:inexistence_by_normal} is satisfied.
Therefore all assumptions in Lemma \ref{lem:inexistence_by_normal} are satisfied.
Hence the claim follows from Lemma \ref{lem:inexistence_by_normal}.
\end{proof}

\section{Compressions over Symmetric and Alternating Groups}
\label{sec:over_symmetric_groups}

In this section, we focus on the following special case: $G$ is either a symmetric group $S_n$ or an alternating group $A_n$, and the type of a compression function is of the form $(\sigma; (0,1),(1,\sigma),(2,\sigma))$ where $\sigma \in G$ has order $3$.
See Examples \ref{exmp:over_S_5} and \ref{exmp:over_A_5} for such examples.
For this type, it follows from Theorem \ref{thm:solvable_group} that a compression function over a solvable group does not exist.
As any group of order less than $60$ is solvable, this implies that $A_5$ is the smallest possible underlying group for a compression function of this type.

When we only consider the underlying groups $G = S_5$ and $G = A_5$, any element in $G$ of order $3$ is a cyclic permutation $(a\ b\ c)$ of length $3$; in particular, we may assume by symmetry that $\sigma = (1\ 2\ 3)$.
Now we show that there is no advantage of searching for a compression function of this type over $S_5$ instead of $A_5$.

\begin{proposition}
\label{prop:construction_in_S_5_is_reduced_to_A_5}
Let $n \geq 5$ and $\sigma = (1\ 2\ 3) \in A_n$.
If there exists a compression function of type $(\sigma; (0,1),(1,\sigma),(2,\sigma))$, size $\ell$, and exponent $(e_1,\dots,e_{\ell})$ over $S_n$, then there exists a compression function of type $(\sigma; (0,1),(1,\sigma),(2,\sigma))$, size $\ell$, and exponent $(e_1,\dots,e_{\ell})$ over $A_n$.
\end{proposition}
\begin{proof}
As $A_n \lhd S_n$, $[S_n : A_n] = 2$, and $(4\ 5) \in Z_{S_n}(\sigma) \setminus A_n$, we have $S_n = A_n Z_{S_n}(\sigma)$.
Now the claim follows from Proposition \ref{prop:reduction_to_normal_subgroup} where $G := S_n$ and $N := A_n$.
\end{proof}

For the search for a compression function of this type over $A_5$, first we give the following lower bound for the size of such a compression function.

\begin{theorem}
\label{thm:short_case_over_S_5}
Let $\sigma = (1\ 2\ 3)$.
Then there does not exist a compression function of type $(\sigma; (0,1),(1,\sigma),(2,\sigma))$ and size at most $3$ over $A_5$.
\end{theorem}
\begin{proof}
The inexistence for the case of size at most $2$ follows immediately from Proposition \ref{prop:type_2_size_one}.
Assume for the contrary that such a compression function $F$ of this type, size $3$, and exponent $(e_1,e_2,e_3)$ over $A_5$ exists.
As $\ord(\sigma) = 3$, we may assume without loss of generality that $e_1,e_2,e_3 \in \{1,2\}$ (note that if some $e_i$ is a multiple of $3$, then the situation is reduced to the case of size $2$).
Let $(\tau_1,\tau_2,\tau_3) \in (A_5)^3$ be a solution of the system of equations \eqref{eq:lem:restating_with_conjugate_elements:equation} as in Lemma \ref{lem:restating_with_conjugate_elements} implied by the existence of $F$.
Then we have $\tau_1{}^{e_1} \tau_2{}^{e_2} \tau_3{}^{e_3} = \sigma = \tau_1{}^{2e_1} \tau_2{}^{2e_2} \tau_3{}^{2e_3}$, therefore $\tau_2{}^{-e_2} \tau_1{}^{e_1} \tau_2{}^{2e_2} = \tau_3{}^{-e_3}$, or equivalently
\begin{equation}
\label{eq:thm:short_case_over_S_5}
\tau_2{}^{-e_2} \tau_1{}^{e_1} \tau_2{}^{e_2} = \tau_3{}^{-e_3} \tau_2{}^{-e_2} \enspace.
\end{equation}
By the conjugacy condition \eqref{eq:lem:restating_with_conjugate_elements:conjugacy}, $\tau_1{}^{e_1}$ is conjugate to $\sigma^{e_1}$ which is a cyclic permutation of length $3$, therefore the left-hand side of \eqref{eq:thm:short_case_over_S_5}, which is conjugate to $\tau_1{}^{e_1}$, is also a cyclic permutation of length $3$.
Moreover, both $\nu_1 := \tau_3{}^{-e_3}$ and $\nu_2 := \tau_2{}^{-e_2}$ are also cyclic permutations of length $3$ by a similar reason.
As $\nu_1,\nu_2 \in A_5$, we may write $\nu_1 = (a\ b_1\ b_2)$ and $\nu_2 = (a\ c_1\ c_2)$ where $a, b_1, b_2$ are all different and $a, c_1, c_2$ are all different.

If $\{b_1,b_2\} \cap \{c_1,c_2\} \neq \emptyset$, then there is a subgroup $H$ of $S_5$ satisfying that $H \simeq S_4$ and $\nu_1,\nu_2 \in H$.
Now the equality \eqref{eq:thm:short_case_over_S_5} implies that $\tau_1{}^{e_1} = \nu_2{}^{-1} \nu_1 \nu_2{}^2 \in H$.
Therefore, as each $e_i$ is in $\{1,2\}$ and each $\tau_i$ is order $3$ by the conjugacy condition \eqref{eq:lem:restating_with_conjugate_elements:conjugacy}, we have $\tau_1 = (\tau_1{}^{e_1})^{e_1} \in H$, $\tau_2 = \nu_2{}^{-e_2} \in H$, and $\tau_3 = \nu_1{}^{-e_3} \in H$.
Therefore we have $\sigma = \tau_1{}^{e_1} \tau_2{}^{e_2} \tau_3{}^{e_3} \in H$.
Now as all of $\sigma,\tau_1,\tau_2,\tau_3$ are elements of $H \simeq S_4$ of order $3$, those elements are conjugate in $H$ to each other.
Therefore $(\tau_1,\tau_2,\tau_3)$ is a solution of the system of equations \eqref{eq:lem:restating_with_conjugate_elements:equation} over $H$ satisfying the conjugacy condition \eqref{eq:lem:restating_with_conjugate_elements:conjugacy}.
Hence by Lemma \ref{lem:restating_with_conjugate_elements}, there exists a compression function of type $(\sigma; (0,1),(1,\sigma),(2,\sigma))$ over the solvable group $H \simeq S_4$.
This contradicts Theorem \ref{thm:solvable_group}.

From now, we consider the other case where $\{b_1,b_2\} \cap \{c_1,c_2\} = \emptyset$.
Now $\nu_1 \nu_2$, which is the element \eqref{eq:thm:short_case_over_S_5}, has to be a cyclic permutation of length $3$ as mentioned above, while $\nu_1 \nu_2 (c_2) = \nu_1 (a) = b_1$, $\nu_1 \nu_2 (b_1) = \nu_1 (b_1) = b_2$, and $\nu_1 \nu_2 (b_2) = \nu_1 (b_2) = a \neq c_2$.
This is a contradiction.
Hence the claim holds.
\end{proof}

On the other hand, when we search for a compression function of this type and exponent $(e_1,\dots,e_{\ell})$, we may assume without loss of generality that each $e_i$ is in $\{1,2\}$ as $\ord(\sigma) = 3$.
Now we will show that we may moreover assume that each $e_i$ is $1$.
In the proof, the following known fact is essential; we give a proof of this fact for the sake of completeness.

\begin{lemma}
\label{lem:conjugate_in_A_n}
Let $n \geq 5$.
Then any two cyclic permutations of length $3$ are conjugate in $A_n$.
\end{lemma}
\begin{proof}
Let $\rho$ and $\nu$ be cyclic permutations of length $3$ in $A_n$.
Then $\rho$ and $\nu$ are conjugate in $S_n$; say, $\nu = u \rho u^{-1}$ with $u \in S_n$.
Now as $n \geq 5$, there is a transposition $\tau = (a\ b) \in S_n$ with $\tau \rho = \rho \tau$.
Moreover, as $[S_n : A_n] = 2$ and $\tau \not\in A_n$, we have $u = v$ or $u = v \tau$ for some $v \in A_n$.
This implies that $\nu = u \rho u^{-1} = v \rho v^{-1}$ in any case.
Therefore the claim holds.
\end{proof}

\begin{proposition}
\label{prop:exponent_is_1}
Let $n \geq 5$, $\sigma = (1\ 2\ 3)$, $\ell \in \ZZp$, and $e_i \in \{1,2\}$ for $i = 1,\dots,\ell$.
If there exists a compression function $F$ of type $(\sigma; (0,1),(1,\sigma),(2,\sigma))$, size $\ell$, and exponent $(e_1,\dots,e_{\ell})$ over $A_n$, then there exists a compression function of type $(\sigma; (0,1),(1,\sigma),(2,\sigma))$, size $\ell$, and exponent $(1,1,\dots,1)$ over $A_n$.
\end{proposition}
\begin{proof}
Let $(\tau_1,\dots,\tau_{\ell})$ be a solution of the system of equations \eqref{eq:lem:restating_with_conjugate_elements:equation} over $A_n$ as in Lemma \ref{lem:restating_with_conjugate_elements} implied by the existence of $F$.
Then the conjugacy condition \eqref{eq:lem:restating_with_conjugate_elements:conjugacy} implies that each $\tau_i$ is a cyclic permutation of length $3$, therefore $\tau_i{}^{e_i}$ is also a cyclic permutation of length $3$.
Hence by Lemma \ref{lem:conjugate_in_A_n}, $\tau_i{}^{e_i}$ is conjugate to $\sigma$ in $A_n$.
This implies that $(\tau_1{}^{e_1},\dots,\tau_{\ell}{}^{e_{\ell}})$ is a solution of the system of equations \eqref{eq:lem:restating_with_conjugate_elements:equation} over $A_n$ corresponding to the exponent $(1,1,\dots,1)$, satisfying the conjugacy condition \eqref{eq:lem:restating_with_conjugate_elements:conjugacy}.
Therefore the claim follows from Lemma \ref{lem:restating_with_conjugate_elements}.
\end{proof}

Based on the arguments above, we performed a computer search for a compression function of type $(\sigma; (0,1),(1,\sigma),(2,\sigma))$ over $G = A_5$ where $\sigma = (1\ 2\ 3)$.
We set the size of a compression function to be $\ell := 4$, which is the smallest possible value due to Theorem \ref{thm:short_case_over_S_5}.
We focused on the specific exponent $(1,1,1,1)$ owing to Proposition \ref{prop:exponent_is_1}.
Then based on Lemma \ref{lem:restating_with_conjugate_elements}, we searched for a solution $(\tau_1,\tau_2,\tau_3,\tau_4) \in G^4$ of the following system of equations:
\[
\begin{cases}
y_1 y_2 y_3 y_4 = \sigma \enspace,\\
y_1{}^2 y_2{}^2 y_3{}^2 y_4{}^2 = \sigma \enspace.
\end{cases}
\]
Now due to the conjugacy condition \eqref{eq:lem:restating_with_conjugate_elements:conjugacy}, it suffices to search for a solution with the property that each $\tau_i$ is a cyclic permutation of length $3$.
On the other hand, if a solution with each $\tau_i$ being a cyclic permutation of length $3$ is found, then the conjugacy condition \eqref{eq:lem:restating_with_conjugate_elements:conjugacy} is automatically satisfied owing to Lemma \ref{lem:conjugate_in_A_n}.
These properties made the search significantly easier.
Then by a computer search using SageMath, we found a solution $(\tau_1,\tau_2,\tau_3,\tau_4)$ as in Example \ref{exmp:over_A_5_equivalent}, which corresponds to the compression function $F$ as in Example \ref{exmp:over_A_5}.
By the argument above, this is a smallest example of a compression function with the desired property.

\section{Applications to Homomorphic Encryption}
\label{sec:homomorphic_encryption}

In this section, we describe possible applications of compression functions to homomorphic encryption.
First we clarify the definition of homomorphic encryption adopted in this paper; it is slightly more strict than the usual sense since, e.g., here we suppose that the encryption/decryption and homomorphic evaluation in such a scheme must be error-free.

\begin{definition}
\label{defn:HE}
Let $\MM$ be a fixed finite set, and let $\FF$ be a fixed set of operations on $\MM$.
An \emph{$\FF$-homomorphic encryption} ($\FF$-HE) \emph{scheme} with plaintext space $\MM$ is defined to be a tuple $\Pi = (\Gen,\Enc,\Dec,\Eval)$ of (possibly probabilistic) algorithms $\Gen$, $\Enc$, $\Dec$, and $\Eval$ satisfying the following syntax, where each of the four algorithms is required to be of polynomial-time with respect to the security parameter $\lambda$:
\begin{description}
\item[$\Gen(1^{\lambda})$]
Given input $1^{\lambda}$, the key generation algorithm $\Gen$ outputs a tuple of public key $\pk$, secret key $\sk$, and evaluation key $\ek$.
Here $\pk$ involves information on the ciphertext space $\CC$, which is the disjoint union $\CC = \bigsqcup_{m \in \MM} \CC_m$ of subsets $\CC_m$ for $m \in \MM$.
Moreover, we suppose that each of $\sk$ and $\ek$ implicitly involves information on $\pk$. 
\item[$\Enc(\pk,m)$]
Given input $\pk$ and $m \in \MM$, the encryption algorithm $\Enc$ outputs a ciphertext $c \in \CC$. 
\item[$\Dec(\sk,c)$]
Given input $\sk$ and $c \in \CC$, the decryption algorithm $\Dec$ outputs an element of $\MM$.
\item[$\Eval(\ek,f,c_1,\dots,c_n)$]
Given input $\ek$, $f \in \FF$, and $c_1,\dots,c_n \in \CC$, where $f$ is an $n$-ary operation, the evaluation algorithm $\Eval$ outputs an element of $\CC$.
\end{description}
We say that an $\FF$-HE scheme $\Pi$ is \emph{correct} if the following two conditions are satisfied, where $(\pk,\sk,\ek)$ is any output of $\Gen(1^{\lambda})$:
\begin{description}
\item[(Correctness of Encryption/Decryption)]
With probability $1$, we have $\Enc(\pk,m) \in \CC_m$ for any $m \in \MM$, and we have $\Dec(\sk,c) = m$ for any $m \in \MM$ and $c \in \CC_m$.
\item[(Correctness of Evaluation)]
If $f \in \FF$ is an $n$-ary operation, and $m_i \in \MM$ and $c_i \in \CC_{m_i}$ for each $i = 1,\dots,n$, then with probability $1$, we have $\Eval(\ek,f,c_1,\dots,c_n) \in \CC_m$ where $m := f(m_1,\dots,m_n)$.
\end{description}
\end{definition}

In this section, we suppose that any $\FF$-HE scheme under consideration is correct (in the sense above) and secure (in the sense of IND-CPA security; see e.g., \cite{Katz-Lindell} for the definition).

\begin{definition}
\label{defn:FHE}
In the setting of Definition \ref{defn:HE}, we suppose moreover that $\MM = \{0,1\}$.
We say that an $\FF$-HE scheme $\Pi$ is a \emph{fully homomorphic encryption} (FHE) \emph{scheme} with plaintext space $\MM = \{0,1\}$ if $\FF$ is a functionally complete set of operations, that is, any operation on $\MM$ can be realized by a combination of operations in $\FF$.
\end{definition}

It is known that any set $\FF$ of bit operations involving operations $\mathsf{NOT}$ and $\mathsf{OR}$ is functionally complete, and any such set $\FF$ involving operation $\mathsf{NAND}$ is also functionally complete.

Based on the compression function in Example \ref{exmp:over_A_5}, we give the following reduction of a construction of an FHE scheme to a construction of a $\{\cdot_{A_5}\}$-HE scheme where $\cdot_{A_5}$ denotes the multiplication operation on $A_5$.
Our construction of the evaluation algorithm is significantly more efficient than a construction described implicitly in \cite{OstSke08}, and is also more efficient than a similar construction in \cite{Nui21}.
In the following, we write $c_1 \boxdot_{A_5} c_2$ as an abbreviation of $\Eval(\ek,\cdot_{A_5},c_1,c_2)$ for the sake of simplicity, and let the operation $\boxdot_{A_5}$ be left-associative; e.g., $c_1 \boxdot_{A_5} c_2 \boxdot_{A_5} c_3$ means $(c_1 \boxdot_{A_5} c_2) \boxdot_{A_5} c_3$.
Moreover, let $\mathsf{EQ}$ be the $2$-bit equality operation, i.e., $\mathsf{EQ}(b,b') = 1$ if $b = b'$ and $\mathsf{EQ}(b,b') = 0$ if $b \neq b'$; and let $3\mbox{-}\mathsf{NEQ}$ be the $3$-bit non-equality operation, i.e., $3\mbox{-}\mathsf{NEQ}(b,b',b'') = 0$ if $b = b' = b''$ and $3\mbox{-}\mathsf{NEQ}(b,b',b'') = 1$ otherwise.

\begin{theorem}
\label{thm:FHE_from_A_5-HE}
Let $\Pi = (\Gen,\Enc,\Dec,\Eval)$ be any $\{\cdot_{A_5}\}$-HE scheme with plaintext space $\MM = A_5$ and ciphertext space $\CC = \bigsqcup_{m \in A_5} \CC_m$.
Then the following scheme $\widetilde{\Pi} = (\widetilde{\Gen},\widetilde{\Enc},\widetilde{\Dec},\widetilde{\Eval})$ is an $\FF$-HE scheme with plaintext space $\widetilde{\MM} := \{0,1\}$ where $\FF = \{\mathsf{NOT},\mathsf{OR},\mathsf{NAND},\mathsf{XOR},\mathsf{EQ},3\mbox{-}\mathsf{NEQ}\}$, hence $\widetilde{\Pi}$ is an FHE scheme:
\begin{description}
\item[$\widetilde{\Gen}(1^{\lambda})$]
It generates $(\pk,\sk,\ek) \leftarrow \Gen(1^{\lambda})$, generates $\widehat{c}_{\tau} \leftarrow \Enc(\pk,\tau)$ for each $\tau \in S$ where
\[
S := \{(1\ 2\ 3), (1\ 2\ 4\ 3\ 5), (1\ 3\ 5), (1\ 4\ 3), (1\ 5)(2\ 3), (1\ 4\ 3\ 5\ 2)\} \enspace,
\]
and outputs $\widetilde{\pk} := \pk$, $\widetilde{\sk} := \sk$, and $\widetilde{\ek} := (\ek, (\widehat{c}_{\tau})_{\tau \in S})$.
The ciphertext space of $\widetilde{\Pi}$ is $\widetilde{\CC} = \widetilde{\CC}_0 \sqcup \widetilde{\CC}_1$ where $\widetilde{\CC}_0 := \CC_1$ and $\widetilde{\CC}_1 := \CC_{\sigma}$ with $\sigma := (1\ 2\ 3) \in A_5$.
\item[$\widetilde{\Enc}(\widetilde{\pk},m)$]
If $m = 0$, then it generates $c \leftarrow \Enc(\pk,1)$; and if $m = 1$, then it generates $c \leftarrow \Enc(\pk,\sigma)$.
Then it outputs $\widetilde{c} := c$.
\item[$\widetilde{\Dec}(\widetilde{\sk},\widetilde{c})$]
It generates $m' \leftarrow \Dec(\sk,\widetilde{c})$, and outputs $0$ if $m' = 1$ and outputs $1$ if $m' = \sigma$.
\item[$\widetilde{\Eval}(\widetilde{\ek},\mathsf{NOT},\widetilde{c}_1)$]
It computes $c := \widehat{c}_{\sigma} \boxdot_{A_5} \widetilde{c}_1 \boxdot_{A_5} \widetilde{c}_1$ and outputs $c$.
\item[$\widetilde{\Eval}(\widetilde{\ek},f,\widetilde{c}_1,\dots,\widetilde{c}_n)$ for $f \in \{\mathsf{OR},\mathsf{NAND},\mathsf{XOR},\mathsf{EQ},3\mbox{-}\mathsf{NEQ}\}$]
We set $n := 2$ if $f \neq 3\mbox{-}\mathsf{NEQ}$ and $n := 3$ if $f = 3\mbox{-}\mathsf{NEQ}$.
First, it computes $c_{\mathsf{in}}$ as follows:
\[
c_{\mathsf{in}} \leftarrow
\begin{cases}
\widetilde{c}_1 \boxdot_{A_5} \widetilde{c}_2 & \mbox{if $f = \mathsf{OR}$} \enspace,\\
\widehat{c}_{\sigma} \boxdot_{A_5} \widetilde{c}_1 \boxdot_{A_5} \widetilde{c}_2 & \mbox{if $f = \mathsf{NAND}$} \enspace,\\
\widetilde{c}_1 \boxdot_{A_5} \widetilde{c}_1 \boxdot_{A_5} \widetilde{c}_2 & \mbox{if $f = \mathsf{XOR}$} \enspace,\\
\widehat{c}_{\sigma} \boxdot_{A_5} \widehat{c}_{\sigma} \boxdot_{A_5} \widetilde{c}_1 \boxdot_{A_5} \widetilde{c}_2 & \mbox{if $f = \mathsf{EQ}$} \enspace,\\
\widetilde{c}_1 \boxdot_{A_5} \widetilde{c}_2 \boxdot_{A_5} \widetilde{c}_3 & \mbox{if $f = 3\mbox{-}\mathsf{NEQ}$} \enspace.
\end{cases}
\]
Secondly, it computes
\[
\begin{split}
c \leftarrow{}& \widehat{c}_{(1\ 2\ 4\ 3\ 5)} \boxdot_{A_5} c_{\mathsf{in}} \boxdot_{A_5} \widehat{c}_{(1\ 3\ 5)} \boxdot_{A_5} c_{\mathsf{in}} \boxdot_{A_5} \widehat{c}_{(1\ 4\ 3)} \\
&\quad \boxdot_{A_5} c_{\mathsf{in}} \boxdot_{A_5} \widehat{c}_{(1\ 5)(2\ 3)} \boxdot_{A_5} c_{\mathsf{in}} \boxdot_{A_5} \widehat{c}_{(1\ 4\ 3\ 5\ 2)} \enspace.
\end{split}
\]
Then it outputs $c$.
\end{description}
\end{theorem}
\begin{proof}
The correctness of encryption/decryption for $\widetilde{\Pi}$ and the security for $\widetilde{\Pi}$ follow immediately from the corresponding properties of $\Pi$.
Therefore it suffices to show the correctness of evaluation for $\widetilde{\Pi}$.

In the following, let $\PT(\gamma)$ denote the plaintext of $\gamma$ as a ciphertext in $\Pi$ (hence $\PT(\gamma) \in A_5$), and let $\widetilde{\PT}(\gamma)$ denote the plaintext of $\gamma$ as a ciphertext in $\widetilde{\Pi}$ (hence $\widetilde{\PT}(\gamma) \in \{0,1\}$).
First, for the case of $f = \mathsf{NOT}$, if $\widetilde{\PT}(\widetilde{c}_1) = 0$, then $\PT(\widetilde{c}_1) = 1$, therefore the output ciphertext $c$ satisfies that $\PT(c) = \sigma \cdot 1 \cdot 1 = \sigma$ and $\widetilde{\PT}(c) = 1$.
On the other hand, if $\widetilde{\PT}(\widetilde{c}_1) = 1$, then $\PT(\widetilde{c}_1) = \sigma$, therefore the output ciphertext $c$ satisfies that $\PT(c) = \sigma \cdot \sigma \cdot \sigma = 1$ and $\widetilde{\PT}(c) = 0$.
Hence we have $\widetilde{\PT}(c) = \mathsf{NOT}(\widetilde{\PT}(\widetilde{c}_1))$ in any case, as desired.

From now, we consider the other case of $f \in \{\mathsf{OR},\mathsf{NAND},\mathsf{XOR},\mathsf{EQ},3\mbox{-}\mathsf{NEQ}\}$.
Put $m_i := \PT(\widetilde{c}_i)$ and $\widetilde{m}_i := \widetilde{\PT}(\widetilde{c}_i)$ for each $i = 1,\dots,n$.
First, we show that the intermediate ciphertext $c_{\mathsf{in}}$ satisfies that $\PT(c_{\mathsf{in}}) = 1$ if $f(\widetilde{m}_1,\dots,\widetilde{m}_n) = 0$ and $\PT(c_{\mathsf{in}}) \in \{\sigma,\sigma^2\}$ if $f(\widetilde{m}_1,\dots,\widetilde{m}_n) = 1$.
Indeed:
\begin{itemize}
\item
When $f = \mathsf{OR}$:
\begin{itemize}
\item
If $(\widetilde{m}_1,\widetilde{m}_2) = (0,0)$, then $f(\widetilde{m}_1,\widetilde{m}_2) = 0$, while $(m_1,m_2) = (1,1)$ and $\PT(c_{\mathsf{in}}) = m_1 \cdot m_2 = 1$.
\item
If $(\widetilde{m}_1,\widetilde{m}_2) \in \{(1,0),(0,1)\}$, then $f(\widetilde{m}_1,\widetilde{m}_2) = 1$, while $(m_1,m_2) \in \{(\sigma,1),(1,\sigma)\}$ and $\PT(c_{\mathsf{in}}) = m_1 \cdot m_2 = \sigma$.
\item
If $(\widetilde{m}_1,\widetilde{m}_2) = (1,1)$, then $f(\widetilde{m}_1,\widetilde{m}_2) = 1$, while $(m_1,m_2) = (\sigma,\sigma)$ and $\PT(c_{\mathsf{in}}) = m_1 \cdot m_2 = \sigma^2$.
\end{itemize}
\item
When $f = \mathsf{NAND}$:
\begin{itemize}
\item
If $(\widetilde{m}_1,\widetilde{m}_2) = (0,0)$, then $f(\widetilde{m}_1,\widetilde{m}_2) = 1$, while $(m_1,m_2) = (1,1)$ and $\PT(c_{\mathsf{in}}) = \sigma \cdot m_1 \cdot m_2 = \sigma$.
\item
If $(\widetilde{m}_1,\widetilde{m}_2) \in \{(1,0),(0,1)\}$, then $f(\widetilde{m}_1,\widetilde{m}_2) = 1$, while $(m_1,m_2) \in \{(\sigma,1),(1,\sigma)\}$ and $\PT(c_{\mathsf{in}}) = \sigma \cdot m_1 \cdot m_2 = \sigma^2$.
\item
If $(\widetilde{m}_1,\widetilde{m}_2) = (1,1)$, then $f(\widetilde{m}_1,\widetilde{m}_2) = 0$, while $(m_1,m_2) = (\sigma,\sigma)$ and $\PT(c_{\mathsf{in}}) = \sigma \cdot m_1 \cdot m_2 = 1$.
\end{itemize}
\item
When $f = \mathsf{XOR}$:
\begin{itemize}
\item
If $(\widetilde{m}_1,\widetilde{m}_2) = (0,0)$, then $f(\widetilde{m}_1,\widetilde{m}_2) = 0$, while $(m_1,m_2) = (1,1)$ and $\PT(c_{\mathsf{in}}) = m_1 \cdot m_1 \cdot m_2 = 1$.
\item
If $(\widetilde{m}_1,\widetilde{m}_2) = (1,0)$, then $f(\widetilde{m}_1,\widetilde{m}_2) = 1$, while $(m_1,m_2) = (\sigma,1)$ and $\PT(c_{\mathsf{in}}) = m_1 \cdot m_1 \cdot m_2 = \sigma^2$.
\item
If $(\widetilde{m}_1,\widetilde{m}_2) = (0,1)$, then $f(\widetilde{m}_1,\widetilde{m}_2) = 1$, while $(m_1,m_2) = (1,\sigma)$ and $\PT(c_{\mathsf{in}}) = m_1 \cdot m_1 \cdot m_2 = \sigma$.
\item
If $(\widetilde{m}_1,\widetilde{m}_2) = (1,1)$, then $f(\widetilde{m}_1,\widetilde{m}_2) = 0$, while $(m_1,m_2) = (\sigma,\sigma)$ and $\PT(c_{\mathsf{in}}) = m_1 \cdot m_1 \cdot m_2 = 1$.
\end{itemize}
\item
When $f = \mathsf{EQ}$:
\begin{itemize}
\item
If $(\widetilde{m}_1,\widetilde{m}_2) = (0,0)$, then $f(\widetilde{m}_1,\widetilde{m}_2) = 1$, while $(m_1,m_2) = (1,1)$ and $\PT(c_{\mathsf{in}}) = \sigma \cdot \sigma \cdot m_1 \cdot m_2 = \sigma^2$.
\item
If $(\widetilde{m}_1,\widetilde{m}_2) \in \{(1,0),(0,1)\}$, then $f(\widetilde{m}_1,\widetilde{m}_2) = 0$, while $(m_1,m_2) \in \{(\sigma,1),(1,\sigma)\}$ and $\PT(c_{\mathsf{in}}) = \sigma \cdot \sigma \cdot m_1 \cdot m_2 = 1$.
\item
If $(\widetilde{m}_1,\widetilde{m}_2) = (1,1)$, then $f(\widetilde{m}_1,\widetilde{m}_2) = 1$, while $(m_1,m_2) = (\sigma,\sigma)$ and $\PT(c_{\mathsf{in}}) = \sigma \cdot \sigma \cdot m_1 \cdot m_2 = \sigma$.
\end{itemize}
\item
When $f = 3\mbox{-}\mathsf{NEQ}$:
\begin{itemize}
\item
If $(\widetilde{m}_1,\widetilde{m}_2,\widetilde{m}_3) = (0,0,0)$, then $f(\widetilde{m}_1,\widetilde{m}_2,\widetilde{m}_3) = 0$, while $(m_1,m_2,m_3) = (1,1,1)$ and $\PT(c_{\mathsf{in}}) = m_1 \cdot m_2 \cdot m_3 = 1$.
\item
If $(\widetilde{m}_1,\widetilde{m}_2,\widetilde{m}_3) \in \{(1,0,0),(0,1,0),(0,0,1)\}$, then $f(\widetilde{m}_1,\widetilde{m}_2,\widetilde{m}_3) = 1$, while $(m_1,m_2,m_3) \in \{(\sigma,1,1),(1,\sigma,1),(1,1,\sigma)\}$ and $\PT(c_{\mathsf{in}}) = m_1 \cdot m_2 \cdot m_3 = \sigma$.
\item
If $(\widetilde{m}_1,\widetilde{m}_2,\widetilde{m}_3) \in \{(1,1,0),(1,0,1),(0,1,1)\}$, then $f(\widetilde{m}_1,\widetilde{m}_2,\widetilde{m}_3) = 1$, while $(m_1,m_2,m_3) \in \{(\sigma,\sigma,1),(\sigma,1,\sigma),(1,\sigma,\sigma)\}$ and $\PT(c_{\mathsf{in}}) = m_1 \cdot m_2 \cdot m_3 = \sigma^2$.
\item
If $(\widetilde{m}_1,\widetilde{m}_2,\widetilde{m}_3) = (1,1,1)$, then $f(\widetilde{m}_1,\widetilde{m}_2,\widetilde{m}_3) = 0$, while $(m_1,m_2,m_3) = (\sigma,\sigma,\sigma)$ and $\PT(c_{\mathsf{in}}) = m_1 \cdot m_2 \cdot m_3 = 1$.
\end{itemize}
\end{itemize}
Secondly, by Example \ref{exmp:over_A_5}, the output ciphertext $c$ satisfies that $\PT(c) = 1$ if $\PT(c_{\mathsf{in}}) = 1$ and $\PT(c) = \sigma$ if $\PT(c_{\mathsf{in}}) \in \{\sigma,\sigma^2\}$.
Summarizing, if $f(\widetilde{m}_1,\dots,\widetilde{m}_n) = 0$, then we have $\PT(c) = 1$ and therefore $\widetilde{\PT}(c) = 0$; and if $f(\widetilde{m}_1,\dots,\widetilde{m}_n) = 1$, then we have $\PT(c) = \sigma$ and therefore $\widetilde{\PT}(c) = 1$.
This implies the correctness of evaluation for $\widetilde{\Pi}$.
This completes the proof.
\end{proof}

For an underlying $\{\cdot_{A_5}\}$-HE scheme as in Theorem \ref{thm:FHE_from_A_5-HE}, a candidate construction of an HE scheme with plaintext space being any finite group was proposed in \cite{GriPon04}, but it was shown in \cite{CBW07} that the scheme is insecure.
To the author's best knowledge, a construction of a $\{\cdot_{A_5}\}$-HE scheme with plausible security has not been given in the literature and is left as an open problem.


\end{document}